\documentclass{birkjour}
\usepackage{amsmath}
\usepackage{amssymb}
\usepackage{color}

\newtheorem{thm}{Theorem}[section]

\newtheorem{lem}[thm]{Lemma}
\newtheorem{prop}[thm]{Proposition}
\theoremstyle{definition}
\newtheorem{defn}[thm]{Definition}
\theoremstyle{remark}

\newtheorem*{ex}{Example}
\numberwithin{equation}{section}

\newcommand{\A}{\mathcal{A}}
\newcommand{\B}{\mathcal{B}}
\newcommand{\T}{\mathcal{T}}

\renewcommand{\O}{\mathcal{O}}
\renewcommand{\o}{o}
\newcommand{\E}{\mathbb{E}}
\renewcommand{\P}{\mathbb{P}}

\newcommand{\TBR}{\mathcal{TBR}}
\newcommand{\SPR}{\mathcal{SPR}}
\newcommand{\rSPR}{r\mathcal{SPR}}

\newcommand{\vertex}{
\color[rgb]{1,0.0,0.0} \put(0,0){\circle*{8}}
\color[rgb]{1,0.1,0.1} \put(-0.25,0.25){\circle*{7}}
\color[rgb]{1,0.2,0.2} \put(-0.5,0.5){\circle*{6}}
\color[rgb]{1,0.4,0.4} \put(-0.75,0.75){\circle*{5}}
\color[rgb]{1,0.6,0.6} \put(-1.0,1.0){\circle*{4}}
\color[rgb]{1,0.8,0.8} \put(-1.25,1.25){\circle*{3}}
\color[rgb]{1,0.9,0.9} \put(-1.5,1.5){\circle*{2}}
\color[rgb]{1,1,1} \put(-1.75,1.75){\circle*{1}}
\color{black} \put(0,0){\circle{8}}
}

\title[Subtree Transfer Operations in Binary Trees]
{Extremal Distances for Subtree Transfer Operations \\ in Binary Trees}

\author[Atkins]{Ross Atkins}
\address{
University of Oxford \\
Department of Statistics \\
1 South Parks Road \\
Oxford OX1 3TG \\
United Kingdom}
\email{ross.atkins@univ.ox.ac.uk}

\author[McDiarmid]{Colin McDiarmid}
\address{
University of Oxford \\
Department of Statistics \\
1 South Parks Road \\
Oxford OX1 3TG \\
United Kingdom}
\email{cmcd@stats.ox.ac.uk}

\usepackage{color}

\begin{document}

\begin{abstract}
Three standard subtree transfer operations for binary trees, used in particular for phylogenetic trees, are: tree bisection and reconnection ($TBR$), subtree prune and regraft ($SPR$) and rooted subtree prune and regraft ($rSPR$). For a pair of leaf-labelled binary trees with $n$ leaves, the maximum number of such moves required to transform one into the other is $n-\Theta(\sqrt{n})$, extending a result of Ding, Grunewald and Humphries. We show that if the pair is chosen uniformly at random, then the expected number of moves required to transfer one into the other is $n-\Theta(n^{2/3})$. These results may be phrased in terms of agreement forests: we also give extensions for more than two binary trees.
\end{abstract}

\maketitle


\section{Introduction}

A standard way to transform one binary tree into another is by performing `subtree-prune-and-regraft' ($SPR$) moves (definitions are given below). These operations are of particular interest for phylogenetic trees. For a pair of binary trees with $n+1$ labelled leaves (including the root), it was shown by Allen and Steele \cite{allen2001} in $2001$ that the maximum number of $SPR$ moves required to change one into the other, $D_{SPR}(n)$, is between $\frac{n}{2}-\o(n)$ and $n-2$. Martin and Thatte~\cite{martin} show the existence of a common subtree of size $\Omega(\sqrt{\log n})$, which brings the upper-bound for $D_{SPR}(n)$ down to $n - \Omega (\sqrt{ \log n})$.
\\ 
Ding, Grunewald and Humphries \cite{ding} show that $D_{SPR}(n)$ is actually $n-\Theta(\sqrt{n})$. We give an extended version of this result for the rooted and unrooted cases. We also show that if one of the trees is fixed arbitrarily and the other is chosen to maximise the number of SPR moves required to turn one into the other, then still $n-\Theta(\sqrt{n})$ moves are required. We prove the last result by showing that for a fixed pair of binary trees, we can label the leaves in such a way that $n - \Theta(\sqrt{n})$ moves are required to turn one into the other. This result is set in a more general context of agreement forests for sets of $k \geq 2$ trees. We also show that if two trees chosen independently, uniformly at random, then the expected number of moves required to transform one into the other is $n-\Theta(n^{2/3})$ [Theorem \ref{thm:expectation}]. 

Subtree transfer operations are subtly different when acting on rooted trees as opposed to unrooted trees. Some papers have considered $SPR$ and $rSPR$ moves as acting on different classes of binary tree. 
We give a unified treatment here; when dealing with rooted-subtree-prune-and-regraft ($rSPR$) moves, we insist that one of the leaf-labels be $0$, and this leaf acts effectively as the root of the tree. We insist that the root has degree $1$ and it behaves in exactly the same manner as the other leaves for $SPR$ and $TBR$ moves.\footnote{Sometimes, the root of a rooted-binary-tree is defined to have degree $2$. This is equivalent to our definition (see Definition \ref{def:B}). Simply add a new leaf, adjacent to the degree-$2$ root, label it $0$ and make it the new root.}

In the subsections below we define binary trees, the three subtree transfer operations $TBR$, $SPR$ and $rSPR$, and the corresponding metrics they induce over the class of binary trees.

\subsection{Binary Trees}

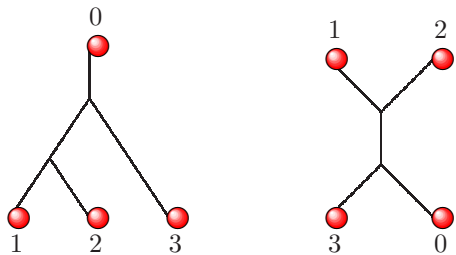
\begin{figure} 
\setlength{\unitlength}{1pt} 
\begin{center}
\begin{picture}(160,80)
	\qbezier(30,65)(30,55)(30,45)
	\qbezier(0,0)(15,22.5)(30,45)
	\qbezier(30,0)(22.5,11.25)(15,22.5)
	\qbezier(60,0)(45,22.5)(30,45)
	\put(0,0){\vertex}
	\put(30,0){\vertex}
	\put(60,0){\vertex}
	\put(30,65){\vertex}
	\put(0,-13){\mbox{$1$}}
	\put(30,-13){\mbox{$2$}}
	\put(60,-13){\mbox{$3$}}
	\put(30,73){\mbox{$0$}}
\put(120,0){
	\qbezier(20,40)(20,30)(20,20)
	\qbezier(0,0)(10,10)(20,20)
	\qbezier(40,0)(30,10)(20,20)
	\qbezier(0,60)(10,50)(20,40)
	\qbezier(40,60)(30,50)(20,40)
	\put(0,0){\vertex}
	\put(40,0){\vertex}
	\put(0,60){\vertex}
	\put(40,60){\vertex}
	\put(0,-13){\mbox{$3$}}
	\put(40,-13){\mbox{$0$}}
	\put(0,68){\mbox{$1$}}
	\put(40,68){\mbox{$2$}}
}
\end{picture}
\end{center}
\caption{Two diagrams for the same tree in $\B(\{ 0,1,2,3 \} )$.}
\label{fig:binaryTree}
\end{figure}

\begin{defn} \label{def:B}
Let $X$ be a non-empty, finite set. If $|X|>1$, then a \emph{binary leaf-labelled tree} with leaf-label set $X$ is a finite tree which has $2$ types of vertex: 
	\begin{itemize}
	\item \emph{tree} vertices of degree $3$, and
	\item \emph{leaf} vertices of degree $1$, labelled with the set $X$ (there is a bijection between the leaves and $X$). 
	\end{itemize}
If $0 \in X$, the leaf labelled $0$ is called the \emph{root} of the tree. In the trivial case, when $X = \{ \alpha \}$ is a singleton, then the trivial graph containing only one vertex, labelled $\alpha$, is considered to be a binary leaf-labelled tree, and its vertex is called a leaf.
\end{defn}

Let $\B(X)$ be the set of all binary leaf-labelled trees, with label set $X$. For a leaf-labelled tree $A$, let $L(A)$ denote the label set of $A$; so if $A \in \B(X)$ then $L(A) = X$. For trees $T$ and $T^\prime$ in $\B(X)$, we say $T$ is isomorphic to $T^\prime$ (denoted $T \equiv T^\prime$) if there is a graph isomorphism preserving leaf labels; if there is a bijection $\psi : V(T) \rightarrow V(T^\prime)$ such that $uv$ is an edge in $T$ if and only if $\psi(x)\psi(y)$ is an edge in $T^\prime$ and for each leaf $v$ of $T$, $v$ and $\psi(v)$ are labelled with the same element of $X$.
For non-empty $S \subseteq L(A)$, let $A|S$ be the minimal subtree of $A$ that contains all the leaves with labels in $S$. Let $A/S$ be the tree formed by suppressing any vertices of degree $2$ from $A|S$. So 
	$$ A|S \mbox{ is a subgraph of } A \qquad \mbox{ and } \qquad A/S \in \B(S). $$
In this paper, $X$ will usually be $\{ 0,1,2, \ldots ,n \}$, i.e. $|X| = n+1$. So $|\B(X)|$ is $1,1,3,15, \ldots$ for $n = 1,2,3,4, \ldots$ and in general for $n \geq 2$, the number is given by 
	$$ |\B(X)| = 1 \times 3 \times 5 \times \cdots \times \left( 2n-3 \right). $$

\subsection{Tree Bisection and Reconnection}

\begin{figure}
\setlength{\unitlength}{1pt} 
\begin{center}
\begin{picture}(200,120)(-10,0)
\put(-10,0){
	\qbezier(30,30)(30,20)(30,10)
	\qbezier(30,30)(45,30)(60,30)
	\qbezier(30,80)(30,90)(30,100)
	\qbezier(30,80)(55,80)(80,80)
	\qbezier(60,80)(60,90)(60,100)
	\put(-10,17){$0$}
	\put(28,-7){$1$}
	\put(63,17){$2$}
	\put(-11,82){$3$}
	\put(22,106){$4$}
	\put(60,107){$5$}
	\put(83,85){$6$}
	\color[rgb]{0.0,0.5,0.0} \linethickness{0.5mm} 
	\qbezier(0,80)(7.5,72.5)(15,65)
	\qbezier(30,80)(22.5,72.5)(15,65)
	\qbezier(0,30)(7.5,37.5)(15,45)
	\qbezier(30,30)(22.5,37.5)(15,45)
	\put(0,30){\vertex}
	\put(30,10){\vertex}
	\put(60,30){\vertex}
	\put(0,80){\vertex}
	\put(30,100){\vertex}
	\put(60,100){\vertex}
	\put(80,80){\vertex}
	\color{blue} \linethickness{1mm}
	\qbezier(15,45)(15,55)(15,65)
}
\put(140,0){
	\qbezier(60,30)(52.5,37.5)(45,45)
	\qbezier(30,30)(37.5,37.5)(45,45)
	\qbezier(30,30)(30,20)(30,10)
	\qbezier(60,80)(52.5,72.5)(45,65)
	\qbezier(30,80)(37.5,72.5)(45,65)
	\qbezier(30,80)(30,90)(30,100)
	\qbezier(60,80)(70,80)(80,80)
	\qbezier(60,80)(60,90)(60,100)
	\put(-10,17){$0$}
	\put(28,-7){$1$}
	\put(63,17){$2$}
	\put(-11,82){$3$}
	\put(22,106){$4$}
	\put(60,107){$5$}
	\put(83,85){$6$}
	\color[rgb]{0.0,0.5,0.0} \linethickness{0.5mm}
	\qbezier(30,80)(15,80)(0,80)
	\qbezier(30,30)(15,30)(0,30)
	\put(0,30){\vertex}
	\put(30,10){\vertex}
	\put(60,30){\vertex}
	\put(0,80){\vertex}
	\put(30,100){\vertex}
	\put(60,100){\vertex}
	\put(80,80){\vertex}
	\color{blue} \linethickness{1mm}
	\qbezier(45,45)(45,55)(45,65)
}
\put(95,50){\mbox{$\Longrightarrow$}}
\end{picture}
\end{center}
\caption{The TBR operation.}
\label{fig:tbrDefinition}
\end{figure}
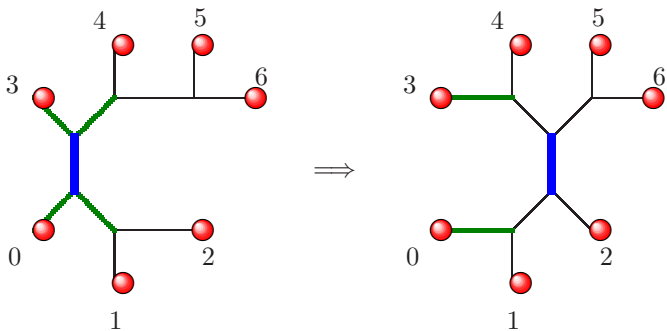

\begin{defn} \label{def:tbr}
A tree-bisection-and-reconnection (TBR) on a tree $T \in \B(X)$, for $|X| \geq 3$, is a two step process:
\begin{enumerate}
\item \emph{(bisection step)} Delete an edge $e$ from $T$ (this edge is called the \emph{bisection} edge) and then suppress all degree $2$ vertices (the edges created by these suppressions are called \emph{new} edges) to obtain a pair $T_1,T_2$ of binary trees with $L(T_1) \sqcup L(T_2) = X$. (The number of new edges is $2$, unless the bisection edge was incident with a leaf. In the later case, there is only one new edge and either $T_1$ or $T_2$ is an isolated vertex.) 
\item \emph{(reconnection step)} Connect $T_1$ and $T_2$ by creating an edge between the midpoint of an edge in $T_1$ and the midpoint of an edge in $T_2$. This edge is called the \emph{reconnecting} edge. (In the case that $T_1$ or $T_2$ is an isolated vertex, instead the reconnecting edge is incident to this vertex.) 
\end{enumerate}
\end{defn}

See Figure~\ref{fig:tbrDefinition}: on the left the bisection edge is blue and the green edges will form the new edges, and on the right the blue edge is the reconnection edge and the green edges are the new edges.

If $A \in \B(X)$, and $B$ is the result of performing a TBR on $A$, then $B \in \B(X)$ because $B$ must be connected, acyclic, and all non-leaf vertices have degree $3$. 

\begin{defn}
Let $\TBR = \TBR(X)$ denote the set of pairs $(A,B) \in \B(X)^2$ such that $B$ can be obtained from $A$ by performing a single TBR move.
\end{defn}

\begin{prop}
\label{prop:sym1}
$(A,B) \in \TBR$ if and only if $(B,A) \in \TBR$.
\end{prop}
\begin{proof}
If $f$ is a TBR which transforms $A$~into~$B$, then we can find a TBR, $f^{-1}$ which transforms $B$~into~$A$ using the following construction. Firstly use the reconnecting edge of $f$ as the bisection edge of $f^{-1}$. The two components formed will be identical to the two components formed when the bisection step of $f$ was performed. If $f$ had two new edges, then we can let the reconnecting edge of $f^{-1}$ be an edge between the midpoints of these edges. If $f$ had one new edge then one part, say $T_2$ must be an isolated vertex $v$, and so we can let the reconnecting edge of $f^{-1}$ be an edge between $v$ and the midpoint of the edge in $T_1$ which was the new edge of $f$. 
\end{proof}

\subsection{Subtree Prune and Regraft}

The SPR move (Definition \ref{def:spr}) is an operation on binary trees whereby a subtree is removed from one part of the tree and regrafted to another part of the tree. SPR moves feature in many papers, see for example: \cite{allen2001} and \cite{bordewich}. These SPR moves are widely used by tree-searching software packages, like PAUP \cite{swofford1998} and Garli \cite{zwickl}, used in phylogenetic research. 

\begin{figure}
\begin{center}
\begin{picture}(230,60)
\put(0,0){
	\qbezier(0,0)(7.5,15)(15,30)
	\qbezier(15,30)(15,45)(15,60)
	\qbezier(45,60)(52.5,30)(60,0)
	\qbezier(30,0)(22.5,15)(15,30)
	\put(0,0){\vertex}
	\put(15,60){\vertex}
	\put(15,30){\vertex}
	\put(45,60){\vertex}
	\put(30,0){\vertex}
	\put(60,0){\vertex}
	\put(10,68){\mbox{$a$}}
	\put(-5,-14){\mbox{$b$}}
	\put(45,68){\mbox{$c$}}
	\put(60,-14){\mbox{$d$}}
	\put(20,25){\mbox{$x$}}
	\put(28,-14){\mbox{$y$}}
}
\put(108,30){\mbox{$\Longrightarrow$}}
\put(170,0){
	\qbezier(0,0)(7.5,30)(15,60)
	\qbezier(45,30)(45,45)(45,60)
	\qbezier(45,30)(52.5,15)(60,0)
	\qbezier(30,0)(37.5,15)(45,30)
	\put(0,0){\vertex}
	\put(15,60){\vertex}
	\put(45,30){\vertex}
	\put(45,60){\vertex}
	\put(60,0){\vertex}
	\put(30,0){\vertex}
	\put(10,68){\mbox{$a$}}
	\put(-5,-14){\mbox{$b$}}
	\put(45,68){\mbox{$c$}}
	\put(60,-14){\mbox{$d$}}
	\put(34,25){\mbox{$x$}}
	\put(28,-14){\mbox{$y$}}	
}
\end{picture}
\end{center}
\caption{The SPR move}
\label{fig:sprDefinition}
\end{figure}
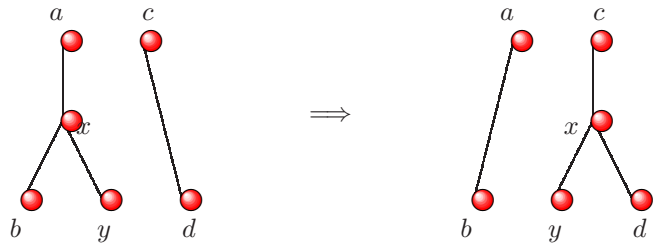

\begin{defn} \label{def:spr}
A \emph{subtree-prune-and-regraft} (SPR) move taking $xy$ from $ab$ to $cd$ is an operation which can be performed on a tree $T \in \B(X)$, as long as $xy,ax,xb,cd$ are distinct edges (vertices $a,b,c,d$ need not be distinct; possibly $a=c$, $b=c$, $a=d$ or $b=d$) and the path from $c$ to $y$ contains $x$. A SPR move is a two step process:
\begin{itemize}
\item \emph{(prune step)} delete edges $ax$, $xb$ and replace them with $ab$, then 
\item \emph{(regraft step)} delete edge $cd$ and replace it with $cx$ and $xd$.
\end{itemize}
Let $\SPR = \SPR(X)$ denote the set of pairs $(A,B)$ such that $A \in \B(X)$ and $B$ is the result of a SPR move on $A$. 
If $0 \in X$ and the path from $0$ to $y$ passes through $x$, then this is called a \emph{rooted-subtree-prune-and-regraft} (rSPR) move. 
Let $\rSPR$ denote the set of pairs $(A,B)$ such that $B$ is the result of an rSPR move on $A$.
\end{defn}

See Figure ~\ref{fig:sprDefinition}. By definition, an rSPR move is a special case of a SPR move. It is not difficult to check that a SPR move is a special case of a TBR move. Informally, a SPR move is any TBR move in which one of $T_1$ or $T_2$ is an isolated vertex or one of the endpoints of the reconnecting edge is the midpoint of one of the new edges. Therefore any valid SPR move performed on a binary tree in $\B(X)$ will result in a tree in $\B(X)$, and moreover: 
	$$ \rSPR \subseteq \SPR \subseteq \TBR. $$
If we think of the edges as being oriented away from the root, then an rSPR move is any SPR move that preserves these orientations.

\begin{prop}
\label{prop:sym2}
If $(A,B) \in \SPR$ then $(B,A) \in \SPR$. Moreover, if $(A,B) \in \rSPR$ then $(B,A) \in \rSPR$.
\end{prop}
\begin{proof} 
Suppose that we start with a tree $A$ and perform the SPR move (or similarly the rSPR move) taking $xy$ from $ab$ to $cd$, and thus obtain $B$. Then starting at $B$, we can perform the `inverse' SPR move (or rSPR move) taking $xy$ from $cd$ to $ab$, and thus we obtain $A$. 
\end{proof}

Not all SPR moves can be achieved with a single rSPR move. Indeed, we may need an arbitrary number of rSPR moves to simulate one SPR move. For example the trees in Figure \ref{fig:rsprnotspr} differ by a single SPR by pruning off the single leaf labelled $0$, and then regrafting it appropriately. In Example~\ref{ex:spr_using_rspr} we will show that at least $\frac{n-3}{2}$ rSPR moves are required to transform one into the other.

\begin{figure}
\begin{center}
\setlength{\unitlength}{1pt} 
\begin{picture}(200,95)(-10,0)
	\qbezier(40,60)(40,67.5)(40,75)
	\qbezier(0,0)(20,30)(40,60)
	\qbezier(20,0)(15,7.5)(10,15)
	\qbezier(40,0)(30,15)(20,30)
	\qbezier(80,0)(60,30)(40,60)
	\put(0,0){\vertex}
	\put(20,0){\vertex}
	\put(40,0){\vertex}
	\put(80,0){\vertex}
	\put(40,75){\vertex}
	\put(0,-13){\mbox{$1$}}
	\put(20,-13){\mbox{$2$}}
	\put(40,-13){\mbox{$3$}}
	\put(56,-2){\mbox{$\cdots$}}
	\put(80,-13){\mbox{$n$}}
	\put(45,83){\mbox{$0$}}
\put(120,0){
	\qbezier(40,60)(40,67.5)(40,75)
	\qbezier(0,0)(20,30)(40,60)
	\qbezier(20,0)(35,22.5)(50,45)
	\qbezier(40,0)(50,15)(60,30)
	\qbezier(80,0)(60,30)(40,60)
	\put(0,0){\vertex}
	\put(20,0){\vertex}
	\put(40,0){\vertex}
	\put(80,0){\vertex}
	\put(40,75){\vertex}
	\put(0,-13){\mbox{$1$}}
	\put(20,-13){\mbox{$2$}}
	\put(40,-13){\mbox{$3$}}
	\put(56,-2){\mbox{$\cdots$}}
	\put(80,-13){\mbox{$n$}}
	\put(45,83){\mbox{$0$}}
}
\end{picture}
\end{center}
\caption{Trees that differ by a single SPR, but at least $\frac{n-3}{2}$ rSPR moves.}
\label{fig:rsprnotspr}
\end{figure}
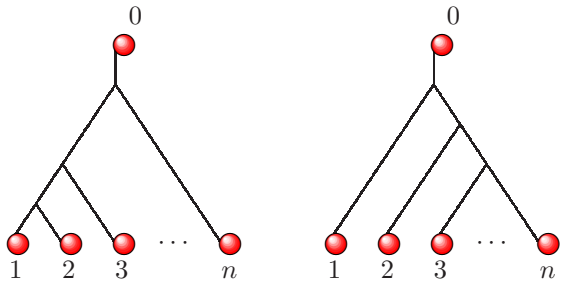

The following proposition is Lemma 2.7 (2b) in \cite{allen2001}.

\begin{prop}
\label{prop:tbrusingsprs}
Any TBR move can be achieved using at most $2$ SPRs; if $(A,B) \in \TBR$, then there exists some $C$ such that $(A,C),(C,B) \in \SPR$.
\end{prop}

\subsection{Tree Metrics}

The $SPR$-distance (Definition \ref{def:distance}) has been the subject of several publications, for example \cite{allen2001}, \cite{golobof}, \cite{hein1990}. The problem of determining the rSPR-distance between a pair of trees was shown to be $NP$-hard by Bordewich and Semple \cite{bordewich} in $2005$. We shall see later, that all of the values in the following definition are finite.

\begin{defn}
For $A,B \in \B(X)$ and any subtree transfer operation $\chi \in \{ TBR , SPR , rSPR \}$, define the $\chi$-distance between $A$ and $B$ (denoted $d_\chi(A,B)$), to be the minimum number of $\chi$ moves required to change $A$ into $B$. When discussing rSPR moves, we assume $0 \in X$. By Propositions \ref{prop:sym1} and \ref{prop:sym2}, this distance is symmetric. For $n \geq 1$ and $X = [n] \cup \{ 0 \}$, let $D_\chi(n)$ denote the $\chi$-\emph{diameter} of the class $\B(X)$; 
	$$ D_\chi(n) \; := \; \max_{A,B \in \B(X)} \; d_\chi(A,B) $$
and moreover, let $R_\chi(n)$ denote the $\chi$-\emph{radius} of the class $\B(X)$; 
	$$ R_\chi(n) \; := \; \min_{C \in \B(X)} \; \max_{B \in \B(X)} \; d_\chi(B,C). $$
\label{def:distance}
\end{defn}

The following is an immediate consequence of $\rSPR \subseteq \SPR \subseteq \TBR$.

\begin{prop}
\label{obs:trivial_bounds}
For any $\chi \in \{ TBR,SPR,rSPR \}$, and any integer $n>1$, we have $D_{\chi}(n) \geq R_{\chi}(n)$. Moreover, 
	$$ D_{rSPR}(n) \geq D_{SPR}(n) \geq D_{TBR}(n) \quad \mbox{and} \quad R_{rSPR}(n) \geq R_{SPR}(n) \geq R_{TBR}(n). $$
\end{prop}

\subsection{Main Theorems}

The values of the radius and diameter for $n \leq 6$ are given in Figure~\ref{tab:small_values}. The asymptotic values of the radius and diameter are given in the following theorem.

\begin{thm}
For each $\chi \in \{ TBR,SPR,rSPR \}$, both $D_\chi(n)$ and $R_\chi(n)$ are $n-\Theta(\sqrt{n})$.
\label{thm:extreme}
\end{thm}

By Observation \ref{obs:trivial_bounds}, it suffices to give an upper bound for the rSPR-diameter and a lower bound for the TBR-radius, both of the form $n-\Theta(\sqrt{n})$. These bounds are given explicitly in Lemmas \ref{lem:ub_D_raf} and \ref{lem:lb_R_uaf} respectively. 

\begin{thm}
\label{thm:expectation}
	If $A$ and $B$ are chosen uniformly at random from $\B( [n] \cup \{ 0 \} )$, then 
	$$ \E[d_\chi(A,B)] = n - \Theta(n^{2/3}) $$
for any $\chi \in \{ TBR,SPR,rSPR \}$.
\end{thm}

Again by Observation \ref{obs:trivial_bounds}, it suffices to give an upper bound for the expected rSPR-distance and a lower bound for the expected TBR-distance, both of the form $n-\Theta(n^{2/3})$. These bounds are given explicitly in Lemmas \ref{lem:ub_expect_raf} and \ref{lem:lb_expect_uaf} respectively.

\begin{figure}
$$ \begin{array}{c|c|c|c|c|c|c}
n & R_{TBR}(n) & D_{TBR}(n) & R_{SPR}(n) & D_{SPR}(n) & R_{rSPR}(n) & D_{rSPR}(n) \\ \hline
2 & 0 & 0 & 0 & 0 & 0 & 0 \\ 
3 & 1 & 1 & 1 & 1 & 1 & 1 \\ 
4 & 2 & 2 & 2 & 2 & 2 & 2 \\ 
5 & 2 & 2 & 2 & 2 & 2 & 3 \\ 
6 & 3 & 3 & 3 & 3 & 3 & 3
\end{array} $$
\caption{Small values of the radius and diameter}
\label{tab:small_values}
\end{figure}

\section{Upper Bound for the rSPR Diameter}

The upper-bound in Theorem~\ref{thm:extreme} is proved in this section using rooted agreement forests. The definition of a rooted agreement forest given here is equivalent to that used by Bordewich and Semple \cite{bordewich}. 

\begin{defn}
\label{def:raf}
Let $0 \in X$, and let $\A \subset \B(X)$ be a set of $k \geq 2$ trees. A \emph{rooted agreement forest} of $\A$ is a partition $L_1, \ldots , L_m$ of $X$, such that for all $A,B \in \A$:
	\begin{itemize}
	\item the subtrees $A|L_1$, $A|L_2$, $\ldots$ , $A|L_m$ are disjoint, and 
	\item $A/(L_j \cup \{ 0 \}) \equiv B/(L_j \cup \{ 0 \})$ for $j = 1, \ldots ,m$.
	\end{itemize}
Let $M_r(\A)$ be the minimal possible value of $m$; the minimal number of parts in a rooted agreement forest of $\A$. When $k=2$, we write $M_r(A,B)$ for $M_r(\{ A,B \})$.
\end{defn}

When $n \geq 2$ and $X = \{ 0,1, \ldots ,n \}$, we can always construct a trivial rooted agreement forest with $n-1$ parts by setting $L_1 = \{ 0 , 1 , 2 \}$ and $L_i = \{ i+1 \}$ for $i = 2, \ldots ,n-1$. Therefore $M_r(A,B) \leq n-1$ for any $A,B \in \B(X)$. Three minimal agreement forests are depicted in Figure \ref{fig:rafDefinition}.

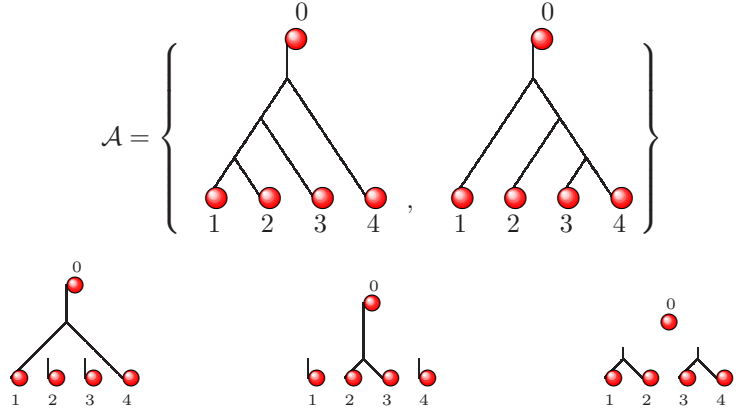
\begin{figure}
\begin{center}
\setlength{\unitlength}{1pt} 
$$ \A = \left\{ 
\begin{picture}(80,40)(-10,20)
	\qbezier(30,60)(30,50)(30,45)
	\qbezier(0,0)(15,22.5)(30,45)
	\qbezier(20,0)(15,7.5)(10,15)
	\qbezier(40,0)(30,15)(20,30)
	\qbezier(60,0)(45,22.5)(30,45)
	\put(0,0){\vertex}
	\put(20,0){\vertex}
	\put(40,0){\vertex}
	\put(60,0){\vertex}
	\put(30,60){\vertex}
	\put(33,67){\mbox{$0$}}
	\put(0,-13){\mbox{$1$}}
	\put(20,-13){\mbox{$2$}}
	\put(40,-13){\mbox{$3$}}
	\put(60,-13){\mbox{$4$}}
\end{picture}
\begin{array}{c} $\;$ \\ $\;$ \\ $\;$ \\ $\;$ \\ , \end{array} 
\begin{picture}(80,40)(-10,20)
	\qbezier(30,60)(30,50)(30,45)
	\qbezier(0,0)(15,22.5)(30,45)
	\qbezier(20,0)(30,15)(40,30)
	\qbezier(40,0)(45,7.5)(50,15)
	\qbezier(60,0)(45,22.5)(30,45)
	\put(0,0){\vertex}
	\put(20,0){\vertex}
	\put(40,0){\vertex}
	\put(60,0){\vertex}
	\put(30,60){\vertex}
	\put(33,67){\mbox{$0$}}
	\put(0,-13){\mbox{$1$}}
	\put(20,-13){\mbox{$2$}}
	\put(40,-13){\mbox{$3$}}
	\put(60,-13){\mbox{$4$}}
\end{picture}
\right\} $$

\vspace{5mm}

\setlength{\unitlength}{0.7pt} 
\begin{picture}(400,50)
\put(0,0){
	\qbezier(30,30)(30,40)(30,50)
	\qbezier(30,30)(15,15)(0,0)
	\qbezier(30,30)(45,15)(60,0)
	\qbezier(20,0)(20,5)(20,10)
	\qbezier(40,0)(40,5)(40,10)
	\put(0,0){\vertex}
	\put(20,0){\vertex}
	\put(40,0){\vertex}
	\put(60,0){\vertex}
	\put(30,50){\vertex}
	\put(33,57){\mbox{\tiny $0$}}
	\put(0,-15){\mbox{\tiny $1$}}
	\put(20,-15){\mbox{\tiny $2$}}
	\put(40,-15){\mbox{\tiny $3$}}
	\put(60,-15){\mbox{\tiny $4$}}
}
\put(160,0){
	\qbezier(30,40)(30,25)(30,10)
	\qbezier(30,10)(25,5)(20,0)
	\qbezier(30,10)(35,5)(40,0)
	\qbezier(0,0)(0,5)(0,10)
	\qbezier(60,0)(60,5)(60,10)
	\put(0,0){\vertex}
	\put(20,0){\vertex}
	\put(40,0){\vertex}
	\put(60,0){\vertex}
	\put(30,40){\vertex}
	\put(33,47){\mbox{\tiny $0$}}
	\put(0,-15){\mbox{\tiny $1$}}
	\put(20,-15){\mbox{\tiny $2$}}
	\put(40,-15){\mbox{\tiny $3$}}
	\put(60,-15){\mbox{\tiny $4$}}
}
\put(320,0){
	\qbezier(10,10)(5,5)(0,0)
	\qbezier(10,10)(15,5)(20,0)
	\qbezier(50,10)(45,5)(40,0)
	\qbezier(50,10)(55,5)(60,0)
	\qbezier(10,10)(10,13)(10,16)
	\qbezier(50,10)(50,13)(50,16)
	\put(0,0){\vertex}
	\put(20,0){\vertex}
	\put(40,0){\vertex}
	\put(60,0){\vertex}
	\put(30,30){\vertex}
	\put(33,37){\mbox{\tiny $0$}}
	\put(0,-15){\mbox{\tiny $1$}}
	\put(20,-15){\mbox{\tiny $2$}}
	\put(40,-15){\mbox{\tiny $3$}}
	\put(60,-15){\mbox{\tiny $4$}}
}
\end{picture}
\end{center}
\caption{Three minimal rooted agreement forests.}
\label{fig:rafDefinition} 
\end{figure}

The following result is proved by Bordewich and Semple \cite{bordewich}, building on the work of Hein et al \cite{hein1996}. We give a proof here for completeness.

\begin{lem}[Rooted Agreement Forest Lemma]
\label{lem:rafl}
	If $A$ and $B$ are any two trees in $\B(X)$, then 
	$$ M_r(A,B) = d_{rSPR}(A,B) + 1. $$
\end{lem}
\begin{proof}
	Let $A=A_0,A_1, \ldots ,A_d=B$ be a sequence with $(A_{i-1},A_i) \in \rSPR$ for all $i = 1,2, \ldots ,d$. If we perform all the \emph{prunes} of these rSPR moves (and none of the \emph{regrafts}), then the result will be a partition of the leaves into a valid (rooted) agreement forest for $A$ and $B$. Therefore 
	$$ M_r(A_0,A_d) \leq d+1 = d_{rSPR}(A_0,A_d)+1. $$
Now we will show that $d_{rSPR}(A,B)+1 \leq M_r(A,B)$, by induction on $m := M_r(A,B)$. For the base case $m=1$, the only agreement forest with one component would be $A=B$ itself, thence $d_{rSPR}(A,B)=0$. For the inductive step $m \geq 2$, it suffices to show that there exists some $A^\prime \in \B(X)$ such that $(A,A^\prime) \in \rSPR$ and $M_r(A^\prime ,B) \leq m-1$. Let $F = (L_1, \ldots ,L_m)$ be an agreement forest for $(A,B)$ of $m \geq 2$ parts and wlog $0 \in L_1$. We can construct $A^\prime$ from $F$ explicitly:
\begin{itemize}
\item Wlog $B|L_2$ is one of the subtrees of $B$ that is connected to $B|L_1$ by a path $p_B$, which does not intersect $B|L_i$ for any $i>2$. 
\item Let $e = yx \in E(A)$ be the first edge on the path $p_A$ from $A|L_2$ to $A|L_1$, where $y$ is in $A|L_2$ and $x$ is not. If $x$ is a leaf (i.e. if $L_1 = \{ 0 \}$ and $x=0$ is the root) then we do not perform an rSPR move. Otherwise $x$ is a tree vertex. Let $a$ and $b$ be the two neighbours of $x$ other than $y$.
\item Now we can prune $A$ at $e$ (that is, take $xy$ from its current position on $ab$) and then regraft it to form $A^\prime$ such that 
	$$A^\prime / (L_1 \cup L_2) \equiv B / (L_1 \cup L_2).$$ 
In the case that $L_1 = \{ 0 \}$ is a singleton, this involves regrafting at the only edge incident with the root of $A$. Otherwise (if $|L_1| \geq 2$) this involves regrafting at an internal edge of $A|L_1$ corresponding to where the path $p_B$ joins $B|L_1$. 
\end{itemize}
We know that $\{ L_1 \cup L_2,L_3,L_4, \ldots , L_m \}$ is an agreement forest for $(A^\prime ,B)$, because of for all $i>2$:
	\begin{itemize}
	\item $A^\prime | (L_1 \cup L_2)$ does not intersect $A^\prime |L_i$,
	\item $B | (L_1 \cup L_2) = B|L_1 \cup p_B \cup B|L_2$ does not intersect $B|L_i$,
	\item and $A^\prime / (L_1 \cup L_2) \equiv B / (L_1 \cup L_2)$. 
	\end{itemize}
Hence $M_r(A^\prime,B) \leq m-1$, which (by the inductive hypothesis) means that $d_{rSPR}(A^\prime,B) \leq m-2$ and therefore $d_{rSPR}(A,B) \leq m-1$ as required.
\end{proof}

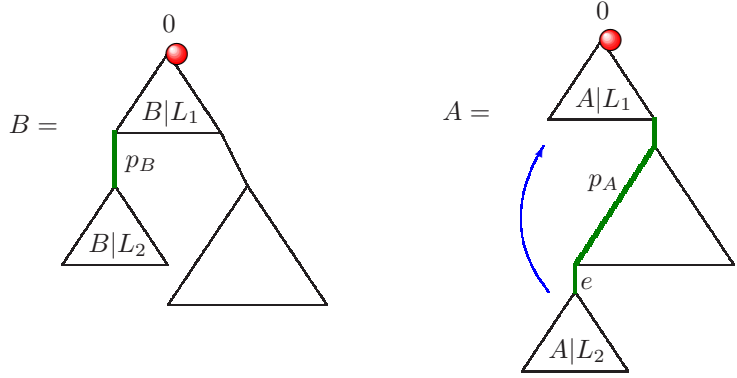
\begin{figure} 
\begin{center}
\setlength{\unitlength}{1pt} 
\begin{picture}(100,130)(0,10)
	\qbezier(40,130)(30,115)(20,100)
	\qbezier(40,130)(50,115)(60,100)
	\qbezier(20,100)(40,100)(60,100)
	\qbezier(20,80)(10,65)(0,50)
	\qbezier(20,80)(30,65)(40,50)
	\qbezier(0,50)(20,50)(40,50)
	\qbezier(60,100)(65,90)(70,80)
	\qbezier(70,80)(55,57.5)(40,35)
	\qbezier(70,80)(85,57.5)(100,35)
	\qbezier(40,35)(70,35)(100,35)
	\put(40,130){\vertex}
	\put(38,138){\mbox{$0$}}
	\put(30,105){\mbox{$B|L_1$}}
	\put(10,55){\mbox{$B|L_2$}}
	\put(-20,100){\mbox{$B=$}}
	\put(24,87){\mbox{$p_B$}}
	\color[rgb]{0.0,0.5,0.0} \linethickness{0.5mm} 
	\qbezier(20,80)(20,90)(20,100)

\end{picture}
\hspace{20mm}
\begin{picture}(90,130)(0,5)
	\qbezier(40,130)(30,115)(20,100)
	\qbezier(40,130)(50,115)(60,100)
	\qbezier(20,100)(40,100)(60,100)
	\qbezier(60,90)(75,67.5)(90,45)
	\qbezier(30,45)(60,45)(90,45)
	\qbezier(30,35)(20,20)(10,5)
	\qbezier(30,35)(40,20)(50,5)
	\qbezier(10,5)(30,5)(50,5)
	\put(40,130){\vertex}
	\put(38,138){\mbox{$0$}}
	\put(30,105){\mbox{$A|L_1$}}
	\put(20,10){\mbox{$A|L_2$}}
	\put(-20,100){\mbox{$A=$}}
	\put(35,75){\mbox{$p_A$}}
	\put(32,36){\mbox{$e$}}
	\color[rgb]{0.0,0.0,1.0} 
	\qbezier(20,35)(0,60)(18,90)
	\put(18,90){\vector(1,2){0}}
	\color[rgb]{0.0,0.5,0.0} \linethickness{0.5mm} 
	\qbezier(60,100)(60,95)(60,90)
	\qbezier(60,90)(45,67.5)(30,45)
	\qbezier(30,45)(30,40)(30,35)
\end{picture}
\end{center}
\caption{The construction in Lemma \ref{lem:rafl}.}
\label{fig:rafl}
\end{figure}

The rSPR-distance between a pair of trees is often used as a measure of the discrepancy between a pair of binary trees. So $M_r(\A)$ can be seen as generalisation of this notion from pairs of trees to arbitrary large sets of trees. 

\begin{ex}
\label{ex:spr_using_rspr}
Let $A$ and $B$ be the two caterpillar trees in Figure \ref{fig:rsprnotspr}, which differ only by a single SPR changing the location of the root. The other leaves in $A$ are ordered in the usual ($1,2,3, \ldots n$) order from the root, while in $B$ they are in the reverse order. We will show here that $M_r(A,B) \geq \frac{n-1}{2}$. By the rooted agreement forest lemma (Lemma \ref{lem:rafl}) this means $d_{rSPR}(A,B) \geq \frac{n-3}{2}$ and so 
	$$ D_{rSPR}(n) \geq \frac{n-3}{2}. $$ 
To show $M_r(A,B) \geq \frac{n-1}{2}$, notice that $A/\{ 0,a,b,c \}$ is never isomorphic to $B/\{ 0,a,b,c \}$ because the leaf neasrest to $0$ in $A/\{ 0,a,b,c \}$ is the smallest of $\{ a,b,c \}$ while the leaf nearest to $0$ in $B/\{ 0,a,b,c  \}$ is the largest. Therefore if $\{ L_1,L_2, \ldots ,L_m\}$ is a rooted agreement forest for $\{ A,B \}$, with $0 \in L_1$, then $|L_1| \leq 3$ and $|L_i| \leq 2$ for all $i>1$. So $n \leq 3+2(m-1)$ and therefore $m \geq \frac{n-1}{2}$.
\end{ex}

The next two lemmas provide an explicit construction to find an agreement forest for arbitrary $\A \subseteq \B(X)$.

\begin{lem}
\label{lem:cut}
Let $|X| = n+1$, let $T \in \B(X)$ be a binary tree with a root and let $a \in (1,2n]$ be a real number. It is possible to remove a single edge from $T$, so that the number of leaves in the component not containing the root is in the interval $\left[ \tfrac{a}{2} , a \right)$.
\end{lem}
\begin{figure} 
\begin{center}
\setlength{\unitlength}{1pt} 
\begin{picture}(140,125)
	\qbezier(70,125)(70,115)(70,105)
	\qbezier(0,0)(35,52.5)(70,105)
	\qbezier(20,0)(15,7.5)(10,15)
	\qbezier(40,0)(65,37.5)(90,75)
	\qbezier(60,0)(65,7.5)(70,15)
	\qbezier(80,0)(70,15)(60,30)
	\qbezier(100,0)(105,7.5)(110,15)
	\qbezier(120,0)(100,30)(80,60)
	\qbezier(140,0)(105,52.5)(70,105)
	\put(0,0){\vertex}
	\put(20,0){\vertex}
	\put(40,0){\vertex}
	\put(60,0){\vertex}
	\put(80,0){\vertex}
	\put(100,0){\vertex}
	\put(120,0){\vertex}
	\put(140,0){\vertex}
	\put(70,125){\vertex}
	\put(75,105){\mbox{$v_0$}}
	\put(93,75){\mbox{$v_1$}}
	\put(68,60){\mbox{$v_2$}}
	\put(48,30){\mbox{$v_3$}}
	\put(73,15){\mbox{$v_4$}}
	\put(58,130){\mbox{$0$}}
\end{picture}
\end{center}
\caption{The construction in Lemma \ref{lem:cut} for $n=8$.}
\label{fig:cut}
\end{figure}
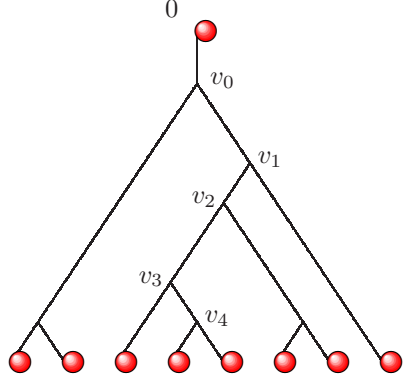
\begin{proof}
Wlog let $0 \in X$ be the root. For each vertex $x \in T$ let $\pi(x)$ be the set of leaves $y$ such that $x$ is on the path between $0$ and $y$. For every non-leaf vertex $p$, 
	$$ \pi(p) = \pi(c_1) \sqcup \pi(c_2), $$
where $c_1,c_2 \in \Gamma(p)$ are the neighbours of $p$ which are not on the path from $p$ to the root. Moreover, if $v$ is a leaf then $|\pi(v)|=1$ (unless $v=0$, in which case $\pi(v)=X$). Therefore, for each tree-node $v_i$, there exists a neighbour $v_{i+1}$ such that 
	\begin{equation}
	\label{eq:half}
	|\pi(v_i)| > |\pi(v_{i+1})| \geq \tfrac{1}{2} |\pi(v_i)|.
	\end{equation}
Now construct a path $(v_0,v_1, \ldots ,v_m)$, such that $v_0$ is the root, $v_1$ is the neighbour of the root and equation~\eqref{eq:half} holds for each $i=1,2, \ldots ,m-1$. This sequence terminates when $v_m$ is a leaf. So the sequence decreases from 
	$$ |\pi(v_1)|=n \geq \frac{a}{2} \quad \mbox{ to } \quad |\pi(v_m)|=1 < a, $$ 
and never decreases by a factor less than $\frac{1}{2}$ in a single step. Hence there must be some $1 \leq j \leq m$ such that 
	$$ \frac{a}{2} \leq |\pi(v_j)| < a. $$
If edge $v_jv_{j-1}$ is deleted, then $\pi(v_j)$ is precisely the set of leaves in the component containing $v_j$.
\end{proof}
\begin{lem} 
\label{lem:divide_tree_equally}
For $|X| = n+1 \geq 2$, let $T \in \B(X)$ and let $a$ be a real number greater than $1$. It is possible to divide $T$ into disjoint subtrees such that each leaf lies in a subtree, each subtree contains less than $a$ leaves and at most one subtree contains less than $\frac{a}{2}$ leaves. 
\end{lem} 
\begin{proof} (by induction)
The case $a>n+1$ is trivial. If $n+1 \geq a$, then we can use Lemma~\ref{lem:cut} iteratively to prune off subtrees with less than $a$ (but at least $\frac{a}{2}$) leaves until the remaining tree has less than $a$ leaves left (this remainder might have less than $\frac{a}{2}$ leaves). 
\end{proof}

\begin{lem}
\label{lem:ub_D_raf}
For $n>k \geq 2$, and any collection $\A = \{ A_1, \ldots ,A_k \} \subseteq \B(X)$, 
	$$ M_r(\A) < n - \frac{1}{2(k+1)} \left( \frac{n+1}{k+1} \right)^{\frac{1}{k}} + 1. $$
\end{lem}

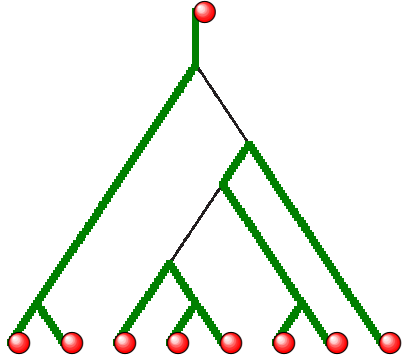
\begin{figure}
\begin{center}
\setlength{\unitlength}{1pt} 
\begin{picture}(140,125)
	\qbezier(40,0)(65,37.5)(90,75)
	\qbezier(140,0)(105,52.5)(70,105)
	\color[rgb]{0.0,0.5,0.0} \linethickness{0.8mm} 
	\qbezier(70,125)(70,115)(70,105)
	\qbezier(0,0)(35,52.5)(70,105)
	\qbezier(20,0)(15,7.5)(10,15)
	\qbezier(40,0)(50,15)(60,30)
	\qbezier(60,0)(65,7.5)(70,15)
	\qbezier(80,0)(70,15)(60,30)
	\qbezier(100,0)(105,7.5)(110,15)
	\qbezier(120,0)(100,30)(80,60)
	\qbezier(140,0)(115,37.5)(90,75)
	\qbezier(80,60)(85,67.5)(90,75)
	\put(0,0){\vertex}
	\put(20,0){\vertex}
	\put(40,0){\vertex}
	\put(60,0){\vertex}
	\put(80,0){\vertex}
	\put(100,0){\vertex}
	\put(120,0){\vertex}
	\put(140,0){\vertex}
	\put(70,125){\vertex}
\end{picture}
\end{center}
\caption{A subdivision of a tree into $3$ disjoint subtrees, each with $3$ leaves.}
\label{fig:divide_tree_equally}
\end{figure}

\begin{proof} (by construction)
Divide each $A_i$ into $t= \left\lfloor \sqrt[k]{\frac{n+1}{k+1}} \right\rfloor$ (possibly empty) disjoint connected parts, 
	$$A_i = A_{i,1} \sqcup A_{i,2} \sqcup \cdots \sqcup A_{i,t},$$ 
such that each part contains less than $\frac{2(n+1)}{t}$ leaves. This is possible by letting $a = \frac{2(n+1)}{t}$ in Lemma \ref{lem:divide_tree_equally}. Now, let us define a \emph{token} to be any map $\tau : [k] \rightarrow [t]$. For any token $\tau$, define 
	$$ I(\tau) := \bigcap_{i=1}^k L(A_{i,\tau(i)}) \qquad \mbox{and} \qquad U(\tau) := \bigcup_{i=1}^k L(A_{i,\tau(i)}). $$ 
We say that $\tau$ is \emph{good} if $I(\tau) \geq 2$, and we say a pair $\tau,\tau^\prime$ is \emph{compatible} if $\tau(i) \not= \tau^\prime(i)$ for all $i$. Now if $T$ is a set of pairwise compatible, good tokens, then we can construct an agreement forest of $n-|T|+1$ parts in the following manner:
\begin{itemize}
\item for each $\tau \in T$, construct a part of size $2$ using labels in $I(\tau)$,
\item all other labels are put in their own part of size $1$.
\end{itemize}
The paths between labels in each part of this partition will be disjoint because the tokens are compatible. 
Hence this is a valid agreement forest. Now let $T$ be any maximal set of compatible good tokens (i.e. every good token that is not in $T$ is incompatible with a token in $T$). Then for each label $x \in X$ either $x \in I(\sigma)$ for some non-good token $\sigma$, or $x \in U(\tau)$ for some $\tau \in T$. So let us count:
	\begin{itemize}
	\item there are $n+1$ labels in total, 
	\item at most $t^k$ non-good tokens, each with at most $1$ label in $I(\sigma)$, and 
	\item less than $ka$ labels in $U(\tau)$ for each token $\tau \in T$.  
	\end{itemize}
Therefore $t^k + ka|T| > n+1$. Substituting $t = \left\lfloor \sqrt[k]{\tfrac{n+1}{k+1}} \right\rfloor$, we can compute 
	$$ |T| > \frac{n+1-t^k}{ka} \geq \frac{n+1 - \left( \tfrac{n+1}{k+1} \right)}{ka} = \frac{t}{2(k+1)}. $$
Since the above inequality is strict and $t,k$ are both integers, we can tighten the bound to:
	$$ |T| \geq \frac{t+1}{2(k+1)} > \frac{1}{2(k+1)} \sqrt[k]{\frac{n+1}{k+1}} $$
as required.
\end{proof}

Setting $k=2$ in Lemma~\ref{lem:ub_D_raf} and then applying the rooted agreement forest lemma~(\ref{lem:rafl}) gives 
	$$ D_{rSPR}(n) < n - \tfrac{1}{6 \sqrt{3}} \sqrt{n}, $$
which yields a suitable upper bound in Theorem~\ref{thm:extreme}.

\section{Upper Bound for the Expectation}

The upper bound in Theorem \ref{thm:expectation} is established in this section. In fact Lemma \ref{lem:ub_expect_raf} is stronger because the underlying unlabelled trees are not chosen randomly.

\begin{lem}
\label{lem:ub_expect_raf}
Let $k \geq 2$ be a fixed integer. Let $X = [n] \cup \{ 0 \}$ and let $\T$ be an arbitrary set of $k$ trees in $\B(X)$. If $\A$ is obtained by performing a permutation of the leaf-labels of each $T \in \T$ (the permutations are chosen independently and uniformly at random), then 
	$$ \E[M_r(\A)] \leq n - \Omega\left( n^{2/(k+1)} \right) $$
as $n \rightarrow \infty$.
\end{lem}

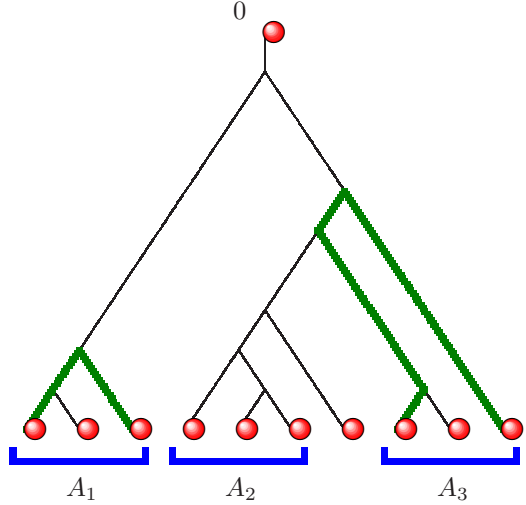
\begin{figure}
\begin{center}
\setlength{\unitlength}{1pt} 
\begin{picture}(140,170)(0,-15)
	\put(58,155){\mbox{$0$}}
	\put(-5,-25){\mbox{$A_1$}}
	\put(55,-25){\mbox{$A_2$}}
	\put(135,-25){\mbox{$A_3$}}
	\qbezier(70,150)(70,145)(70,135)
	\qbezier(-20,0)(25,67.5)(70,135)
	\qbezier(0,0)(-5,7.5)(-10,15)
	\qbezier(40,0)(70,45)(100,90)
	\qbezier(60,0)(65,7.5)(70,15)
	\qbezier(80,0)(70,15)(60,30)
	\qbezier(100,0)(85,22.5)(70,45)
	\qbezier(140,0)(115,37.5)(90,75)
	\qbezier(160,0)(115,67.5)(70,135)
	\color[rgb]{0.0,0.5,0.0} \linethickness{0.8mm} 
	\qbezier(-20,0)(-10,15)(0,30)
	\qbezier(20,0)(10,15)(0,30)
	\qbezier(120,0)(125,7.5)(130,15)
	\qbezier(130,15)(110,45)(90,75)
	\qbezier(160,0)(130,45)(100,90)
	\qbezier(90,75)(95,82.5)(100,90)
	\color[rgb]{0.0,0.0,1.0} \linethickness{0.8mm} 
	\qbezier(-25,-8)(-25,-10)(-25,-12)
	\qbezier(-25,-12)(0,-12)(25,-12)
	\qbezier(25,-8)(25,-10)(25,-12)
	\qbezier(35,-8)(35,-10)(35,-12)
	\qbezier(35,-12)(60,-12)(85,-12)
	\qbezier(85,-8)(85,-10)(85,-12)
	\qbezier(115,-8)(115,-10)(115,-12)
	\qbezier(115,-12)(140,-12)(165,-12)
	\qbezier(165,-8)(165,-10)(165,-12)
	\put(-20,0){\vertex}
	\put(0,0){\vertex}
	\put(20,0){\vertex}
	\put(40,0){\vertex}
	\put(60,0){\vertex}
	\put(80,0){\vertex}
	\put(100,0){\vertex}
	\put(120,0){\vertex}
	\put(140,0){\vertex}
	\put(160,0){\vertex}
	\put(70,150){\vertex}
\end{picture}
\end{center}
\caption{Super pairs lie in differed subtrees $A_i$.}
\label{fig:superPairs} 
\end{figure}

\begin{proof}
Set $b = \left\lfloor \sqrt[k+1]{ \frac{kn^{k-1}}{2} } \right\rfloor$ and $s = \left\lfloor \frac{n+1}{2b} \right\rfloor$. Now for each $T \in \T$, let $T_1, T_2, \ldots ,T_s$ be disjoint subtrees of $T$ such that the number of leaves in each $T_i$ is exactly $b$. This is possible by Lemma~\ref{lem:divide_tree_equally} (set $a=2b$) and taking subtrees as necessary. Let $A \in \A$ be the tree obtained by the random labelling of $T$, and let $A_1,A_2, \ldots ,A_s$ be the subtrees of $A$ corresponding to $T_1,T_2, \ldots ,T_s$ respectively.
Call $\{ x_1, \ldots ,x_j \} \subseteq X$ a \emph{good} set if for each $A \in \A$, the leaves labelled $x_1, \ldots ,x_j$ all lie in the same $A_i$. We say that a good pair of distinct labels $\{ x_1,x_2 \}$ is \emph{spoiled} if for some $A \in \A$, there is another good pair $\{ x_3,x_4 \}$ in the same subtree $A_i$ of $A$ as $x_1$ and $x_2$. We call a pair \emph{super} if it is good but not spoiled. We have 
	$$ \P\big(\{ x_1,x_2 \} \mbox{ is good} \big) = \left( \frac{sb}{n+1} \times \frac{b-1}{n} \right)^k $$ 
because for any $A \in \A$ the probability that $x_1 \in \bigcup_{i=1}^s L(A_i)$ is $\frac{sb}{n+1}$, and given that $x_1 \in L(A_i)$ the probability that $x_2 \in L(A_i)$ is $\frac{b-1}{n}$. Similarly: 
	$$ \P\big( \{ x_1,x_2,x_3 \} \mbox{ is good}\big) = \left( \frac{sb}{n+1} \times \frac{b-1}{n} \times \frac{b-2}{n-1} \right)^k. $$
Since $\{ x_1,x_2,x_3 \}$ being good if and only if $\{ x_1,x_2 \}$ and $\{ x_1,x_3 \}$ are both good, we can see that 
	\begin{equation} 
	\P\big( \{ x_1,x_3 \} \mbox{ good} \big| \{ x_1,x_2 \} \mbox{ good} \big) = \left( \frac{b-2}{n-1} \right)^k \leq \left( \frac{b}{n} \right)^k.
	\label{eq:bound1}
	\end{equation} 
Now fix a good pair $\{ x_1,x_2 \}$ with $\{ x_1,x_2 \} \subseteq L(A_i)$ for some $A \in \A$. If $x_3$ and $x_4$ are chosen from $L(A_i) \setminus \{ x_1,x_2 \}$, we compute an upper bound on the probability that $\{ x_3,x_4 \}$ is a good pair. For each $A^\prime \in \A \setminus A$, the probability that $x_3 \in \bigcup_{j=1}^s L(A_j^\prime )$ is $\frac{sb-2}{n-1}$ and given that $x_3 \in L(A_j^\prime )$, the probability that $x_4 \in L(A_j^\prime )$ is at most $\frac{b-1}{n-2}$. Therefore 
	\begin{equation}
	\P\big( \{ x_3,x_4 \} \mbox{ good} \big| \{ x_1,x_2 \} \mbox{ good} \big) \leq \left( \frac{sb-2}{n-1} \times \frac{b-1}{n-2} \right)^{k-1} \leq \left( \frac{b}{n} \right)^{k-1}. 
	\label{eq:bound2}
	\end{equation}
Now let $Y$ denote the set of all pairs $\{ y_1,y_2 \}$ distinct from $\{ x_1,x_2 \}$, which lie in the same $A_i$ as $\{ x_1,x_2 \}$ for some $A \in \A$. Using the union bound, the probability that $\{ x_1,x_2 \}$ is super is bounded below by 
	\begin{equation} 
	\P\Big( \{x_1,x_2 \} \mbox{ good} \Big) \left[ 1 - \sum_{\{ y_1,y_2 \} \in Y} \P\Big( \{ y_1,y_2 \} \mbox{ good} \Big| \{ x_1,x_2 \} \mbox{ good} \Big) \right] 
	\label{eq:union_bound}
	\end{equation}
where the sum ranges over all $\{ y_1,y_2 \} \in Y$. There are at most $k\binom{b-2}{2}$ such pairs with $\{ x_1,x_2 \} \cap \{ y_1,y_2 \} = \emptyset$, and at most $2k(b-2)$ such pairs with $|\{ x_1,x_2 \} \cap \{ y_1,y_2 \}| = 1$. Substituting Equations (\ref{eq:bound1}) and (\ref{eq:bound2}) into (\ref{eq:union_bound}) yields the lower bound 
	$$ \left( \frac{sb}{n+1} \times \frac{b-1}{n} \right)^k \left[ 1 - k\binom{b-2}{2} \times \left( \frac{b}{n} \right)^{k-1} - 2k(b-2) \times \left( \frac{b}{n} \right)^k \right]. $$
The values $b = \Theta\left(n^{\frac{k-1}{k+1}}\right)$ and $s = \Theta\left( n^{\frac{2}{k+1}} \right)$ were chosen so that the factors in the above expression are positive and so this bound is $\Omega\big( n^{-2k/(k+1)} \big)$. Therefore, if $S$ is the number of super pairs, then 
	$$ \E[S] = \binom{n+1}{2} \P\big( \{ x_1,x_2 \} \mbox{ super} \big) = \Omega\left( n^{2/(k+1)}\right). $$
Now let us partition the labels such that each super pair forms a part of size two (super pairs are disjoint by definition), and all other labels are in individual parts. This partition is a valid agreement forest, because if $\boldsymbol{x}$ and $\boldsymbol{y}$ are any super pairs, then for any $A \in \A$ we know that $A|\boldsymbol{x}$ is contained within $A_i$ and $A|\boldsymbol{y}$ is contained within $A_j$ for some $i \not= j$. Hence $M_r(\A) \leq n+1-S$ and so 
	$$ \E[M_r(\A)] \leq n+1-\E[S] = n - \Omega\big( n^{2/(k+1)} \big). \eqno\qedhere$$
\end{proof}
If we set $k=2$ in Lemma~\ref{lem:ub_expect_raf} and apply the rooted agreement forest lemma (Lemma~\ref{lem:rafl}) we see that 
	$$ \E[d_{rSPR}(A,B)] \leq n - \Omega\big( n^{2/3} \big). $$
If $\T = (A,B)$ is chosen uniformly at random, then a random relabelling of the leaves would not affect the distribution, so the last inequality gives the upper bound in Theorem~\ref{thm:expectation}.

\section{Lower Bound for the Radius}

The lower bound in Theorem \ref{thm:extreme} is established in Lemma~\ref{lem:lb_R_uaf} at the end of this section.
The proofs in this section use a notion similar to the rooted agreement forest from the previous section; Definition \ref{def:uaf} and Lemma \ref{lem:uafl} below are analogous to Definition \ref{def:raf} and Lemma \ref{lem:rafl}, for $TBR$s instead of $rSPR$s. 

\begin{defn}
\label{def:uaf}
Let $\A \subset \B(X)$ be a set of $k \geq 2$ trees. An \emph{unrooted agreement forest} of $\A$ is a partition $L_1, \ldots , L_m$ of $X$, such that for all $A,B \in \A$:
	\begin{itemize}
	\item The subtrees $A|L_1$, $A|L_2$, $\ldots$ , $A|L_m$ are disjoint, and 
	\item $A/L_j \equiv B/L_j$ for $j = 1, \ldots m$.
	\end{itemize}
Let $M(\A)$ be the minimal possible value for $m$; the minimal number of parts in an unrooted agreement forest of $\A$. When $k=2$, we write $M(A,B)$ for $M(\{ A,B \})$.
\end{defn}

\begin{figure}

\begin{center}
\setlength{\unitlength}{1pt} 
$$ \A = \left\{ 
\begin{picture}(80,40)(-10,20)
	\qbezier(30,60)(30,50)(30,45)
	\qbezier(0,0)(15,22.5)(30,45)
	\qbezier(20,0)(15,7.5)(10,15)
	\qbezier(40,0)(30,15)(20,30)
	\qbezier(60,0)(45,22.5)(30,45)
	\put(0,0){\vertex}
	\put(20,0){\vertex}
	\put(40,0){\vertex}
	\put(60,0){\vertex}
	\put(30,60){\vertex}
	\put(28,67){\mbox{$0$}}
	\put(0,-13){\mbox{$1$}}
	\put(20,-13){\mbox{$2$}}
	\put(40,-13){\mbox{$3$}}
	\put(60,-13){\mbox{$4$}}
\end{picture}
\begin{array}{c} $\;$ \\ $\;$ \\ $\;$ \\ $\;$ \\ , \end{array} 
\begin{picture}(80,40)(-10,20)
	\qbezier(30,60)(30,50)(30,45)
	\qbezier(0,0)(15,22.5)(30,45)
	\qbezier(20,0)(30,15)(40,30)
	\qbezier(40,0)(45,7.5)(50,15)
	\qbezier(60,0)(45,22.5)(30,45)
	\put(0,0){\vertex}
	\put(20,0){\vertex}
	\put(40,0){\vertex}
	\put(60,0){\vertex}
	\put(30,60){\vertex}
	\put(28,67){\mbox{$0$}}
	\put(0,-13){\mbox{$1$}}
	\put(20,-13){\mbox{$2$}}
	\put(40,-13){\mbox{$3$}}
	\put(60,-13){\mbox{$4$}}
\end{picture}
\right\} $$

\vspace{5mm}

\setlength{\unitlength}{0.7pt} 
\begin{picture}(60,50)
\put(0,0){
	\qbezier(0,0)(10,10)(20,20)
	\qbezier(20,0)(20,10)(20,20)
	\qbezier(40,0)(40,10)(40,20)
	\qbezier(60,0)(50,10)(40,20)
	\qbezier(20,20)(30,20)(40,20)
	\put(0,0){\vertex}
	\put(20,0){\vertex}
	\put(40,0){\vertex}
	\put(60,0){\vertex}
	\put(30,50){\vertex}
	\put(28,60){\mbox{\tiny $0$}}
	\put(0,-14){\mbox{\tiny $1$}}
	\put(20,-14){\mbox{\tiny $2$}}
	\put(40,-14){\mbox{\tiny $3$}}
	\put(60,-14){\mbox{\tiny $4$}}
}
\end{picture}
\end{center}
\caption{An example of an unrooted agreement forest.}
\label{fig:uafDefinition} 
\end{figure}
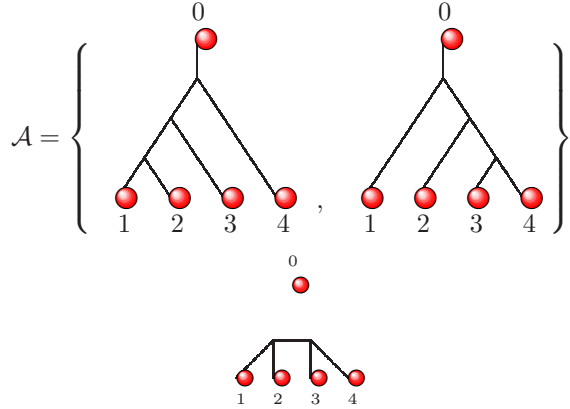

\begin{lem}[Unrooted Agreement Forest Lemma]
\label{lem:uafl}
	If $A$ and $B$ are any two trees in $\B(X)$, then 
	$$ M(A,B) = d_{TBR}(A,B) + 1. $$
\end{lem}
\begin{proof} 
Let $A=A_0,A_1, \ldots A_d=B$, be a sequence satisfying $(A_{i-1},A_i) \in \TBR$ for each $i \in [d]$. If we perform all the \emph{bisections} of these TBR moves (and none of the \emph{reconnections}), then the result will be an agreement forest of at most $d+1$ parts. Therefore 
	$$ M(A_0,A_m) \leq d+1 = d_{TBR}(A_0,A_m) + 1.$$
Now we show $m := M(A,B) \geq d_{TBR}(A,B)+1$, by induction on $m$. For the base case $m=1$, any agreement forest with one component must be $A=B$ itself, thence $d_{TBR}(A,B)=0$. For the inductive step $m \geq 2$, it suffices to show that there exists some $A^\prime$ with $(A,A^\prime) \in \TBR$, such that $M(A^\prime ,B) \leq m-1$. Let $F = \{ L_1, \ldots, L_m \}$ be an agreement forest for $(A,B)$. We can construct $A^\prime$ from $F$ explicitly:
\begin{description}
\item[Step 1:] Choose $e$, the bisection edge, arbitrarily from the edges not in $A|L_i$ for any $i$. We know that such an edge exists because $m \geq 2$.
\end{description}
The deletion of $e$ forms two non-empty components $A_1$ and $A_2$. Note that each of the parts $A|L_i$ lie completely within one of $A_1$ or $A_2$. 
\begin{description}
\item[Step 2:] Now arbitrarily choose a pair of parts $L_x$ and $L_y$, such that $A|L_x$ lies in $A_1$ and $A|L_y$ lies in $A_2$. Let $p_B$ be the path in $B$ from a vertex in $B|L_x$ to a vertex in $B|L_y$.
\item[Step 3:] Let $v_1$ be the last vertex on $p_B$ which lies in a part $B|L_1$ (not necessarily distinct from $B|L_x$) such that $A|L_1$ lies in $A_1$. Let $v_2$ be first vertex after $v_1$ in $p_B$ which lies in a part $B|L_2$ (not necessarily distinct from $B|L_y$). By definition, $A|L_2$ must lie in $A_2$ (see Figure \ref{fig:uafLemma}). 
\end{description}

So the path between $v_1$ and $v_2$ has no vertices in any $B|L_i$ except for its endpoints. Now $v_1$ corresponds to the midpoint of an edge in $B/L_1 \cong A/L_1$, which in turn corresponds to a path $p_1$ in $A_1$ (or else $x=1$ and $L_1$ is a singleton). Similarly $v_2$ corresponds to a $p_2$ in $A_2$ (or $L_2 = \{ v_2 \}$). 
\begin{description}
\item[Step 4:] Construct $A^\prime$ to be the reconnection of $A_1$ and $A_2$ by an edge between the midpoint of an edge in $p_1$ and the midpoint of an edge in $p_2$ (if $L_1$ or $L_2$ is a singleton, then simply use its vertex as the endpoint of the reconnecting edge). 
\end{description}
By construction, $(A,A^\prime) \in \TBR$ because we constructed $A^\prime$ by performing a TBR move on $A$. Moreover $A^\prime / (L_1 \cup L_2) \equiv B/(L_1 \cup L_2)$, so
	$$ F^\prime = \big\{ (L_1 \cup L_2) , L_3, \ldots ,L_m \big\} $$
is an agreement forest for $(A^\prime ,B)$, of size $m-1$. 
\end{proof}

We now present a result concerning permutations (Lemma~\ref{lem:orderings}) and a property of $\A$ that guarantees large $M(\A)$ (Lemma~\ref{lem:deleting_edges}). These two Lemmas are the two crucial steps required for the proof of the main result of this section (Lemma~\ref{lem:lb_R_uaf}). 

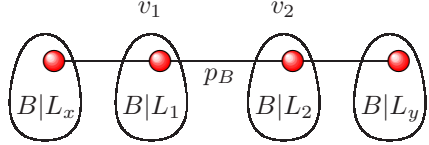
\begin{figure}
\setlength{\unitlength}{1pt} 
\begin{center}
\begin{picture}(130,50)(-10,0)
\qbezier(0,30)(65,30)(130,30)
\put(0,0){
	\qbezier(0,40)(-10,40)(-13,20)
	\qbezier(0,40)(10,40)(13,20)
	\qbezier(0,0)(-16,0)(-13,20)
	\qbezier(0,0)(16,0)(13,20)
	\put(-11,10){\mbox{$B|L_x$}}
	\put(0,30){\vertex}
}
\put(40,0){
	\qbezier(0,40)(-10,40)(-13,20)
	\qbezier(0,40)(10,40)(13,20)
	\qbezier(0,0)(-16,0)(-13,20)
	\qbezier(0,0)(16,0)(13,20)
	\put(-11,10){\mbox{$B|L_1$}}
	\put(0,30){\vertex}
}
\put(90,0){
	\qbezier(0,40)(-10,40)(-13,20)
	\qbezier(0,40)(10,40)(13,20)
	\qbezier(0,0)(-16,0)(-13,20)
	\qbezier(0,0)(16,0)(13,20)
	\put(-11,10){\mbox{$B|L_2$}}
	\put(0,30){\vertex}
}
\put(130,0){
	\qbezier(0,40)(-10,40)(-13,20)
	\qbezier(0,40)(10,40)(13,20)
	\qbezier(0,0)(-16,0)(-13,20)
	\qbezier(0,0)(16,0)(13,20)
	\put(-11,10){\mbox{$B|L_y$}}
	\put(0,30){\vertex}
}
\put(60,23){\mbox{$p_B$}}
\put(35,48){\mbox{$v_1$}}
\put(85,48){\mbox{$v_2$}}
\end{picture}
\end{center}
\caption{The construction in Lemma~\ref{lem:uafl}.}
\label{fig:uafLemma}
\end{figure}

\begin{lem}
\label{lem:orderings}
For any given integers $n,b,k>0$ such that $2 \leq b \leq \sqrt[k]{n+1}$, let $X = \{ 0,1, \ldots ,n \}$. There exist $k$ permutations $\{ \phi_j : X \rightarrow X \}_{j=1}^k$ such that for any distinct $x,y \in X$ there exists some $i$ such that 
	$$ |\phi_i(x)-\phi_i(y)| > b^{k-1} - 2b^{k-3}. $$
\end{lem}

\begin{proof} (by construction) 
Represent each $x \in X$ by the $k$-tuple $(x_1,x_2, \ldots ,x_k)$ of non-negative integers such that 
	\begin{itemize}
	\item $x = x_1b^{k-1} + x_2b^{k-2} + \cdots + x_k$, and 
	\item $0 \leq x_i < b$ for all $1 < i \leq k$. 
	\end{itemize}
If $x < b^k$, then this is exactly representing $x$ in base $b$. For larger $x$, we allow the coefficient of $b^{k-1}$ to exceed $b-1$. Now let $\phi_1$ be the identity permutation, and for each $1 < j \leq k$ we define $\phi_j$ so that 
	$$ \phi_j(x) > \phi_j(y) \quad \mbox{iff} \quad \begin{cases} x_j<y_j & \mbox{or} \\ x_j=y_j \;\; \mbox{and} \;\; x>y. \end{cases} $$
See Figure \ref{fig:orderings} for an illustration. If $d_j = x_j-y_j$ then for $i \geq 2$ and  $\phi_i(x)>\phi_i(y)$ (i.e. $d_i \leq 0$) we have 
	\begin{equation}
	\phi_i(x) - \phi_i(y) \geq -d_ib^{k-1} + \sum_{j=1}^{i-1} d_jb^{k-j-1} + \sum_{j=i+1}^k d_jb^{k-j}. 
	\label{eq:order}
	\end{equation}
This is an equality when $n+1=b^k$ (i.e. when $n+1$ is a perfect $k^{\mbox{\scriptsize th}}$ power), and increasing $n$ can only increase the number of labels $z$ such that $\phi_i(z)$ is between $\phi_i(x)$ and $\phi_i(y)$. \\[2 mm]
Now we fix distinct $x,y \in X$ and wlog $x_1 \geq y_1$. There are four cases: 
\begin{enumerate}
\item $|x_i - y_i| > 1$ for some $i$
\item $x_1=y_1$
\item $x_1>y_1$ and $x_i \geq y_i$ for all $i$. 
\item $x_1>y_1$ and $x_j<y_j$ for some $j \geq 2$.
\end{enumerate}
We treat each of these cases separately.
	\begin{description}
	\item[Case 1:] All $z$ such that $z_i= \min(x_i,y_i)+1$ must have $\phi_i(z)$ between $\phi_i(x)$ and $\phi_i(y)$. There are at least $b^{k-1}$ such $z$. 
	\end{description}
\begin{description}
\item[Case 2:] Wlog $d_i=-1$ and $d_j=0$ for $j<i$ for some $i \geq 2$. Subbing this into equation (\ref{eq:order}) yields
	\begin{align*} \phi_i(x)-\phi_i(y) 
	& \geq b^{k-1} - b^{k-i-1} - b^{k-i-2} - \cdots - 1 \\
	& > b^{k-1} - 2 b^{k-3}. 
	\end{align*}
\item[Case 3:] If $d_1=1$ and $d_j \geq 0$ for all $j$, then $|\phi_1(x) - \phi_1(y)| = x-y \geq b^{k-1}$.
\item[Case 4:] We have $d_1=1$ and there is some $i \geq 2$, $d_i=-1$ and $d_j \geq 0$ for all $j<i$. Substituting this into equation (\ref{eq:order}) yields  
	\begin{align*} \phi_i(x)-\phi_i(y) 	
	& \geq b^{k-1} + b^{k-2} - b^{k-3} - b^{k-4} - \cdots - 1 \\
	& > b^{k-1}. 
	\end{align*}
\end{description}
\end{proof}

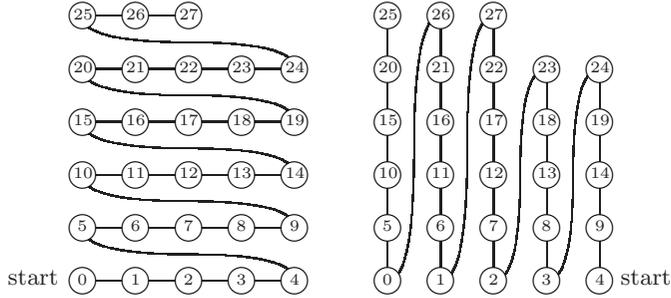
\begin{figure} 
\begin{center}
\setlength{\unitlength}{0.4pt}
\newcommand{\ccc}{\color{white} \put(0,0){\circle*{25}} \color{black} \put(0,0){\circle{25}}}
\begin{picture}(200,250)
\multiput(0,0)(0,50){5}{
	\qbezier(0,0)(100,0)(200,0)
	\qbezier(200,0)(200,25)(100,25)
	\qbezier(100,25)(0,25)(0,50)
	\multiput(0,0)(50,0){5}{\ccc}
}
\qbezier(0,250)(50,250)(100,250)
\multiput(0,250)(50,0){3}{\ccc}
\put(-4.5,-3.2){\mbox{\tiny $0$}}
\put(45.5,-3.2){\mbox{\tiny $1$}}
\put(95.5,-3.2){\mbox{\tiny $2$}}
\put(145.5,-3.2){\mbox{\tiny $3$}}
\put(195.5,-3.2){\mbox{\tiny $4$}}
\put(-4.5,46.8){\mbox{\tiny $5$}}
\put(45.5,46.8){\mbox{\tiny $6$}}
\put(95.5,46.8){\mbox{\tiny $7$}}
\put(145.5,46.8){\mbox{\tiny $8$}}
\put(195.5,46.8){\mbox{\tiny $9$}}
\put(-8.5,96.8){\mbox{\tiny $10$}}
\put(41.5,96.8){\mbox{\tiny $11$}}
\put(91.5,96.8){\mbox{\tiny $12$}}
\put(141.5,96.8){\mbox{\tiny $13$}}
\put(191.5,96.8){\mbox{\tiny $14$}}
\put(-8.5,146.8){\mbox{\tiny $15$}}
\put(41.5,146.8){\mbox{\tiny $16$}}
\put(91.5,146.8){\mbox{\tiny $17$}}
\put(141.5,146.8){\mbox{\tiny $18$}}
\put(191.5,146.8){\mbox{\tiny $19$}}
\put(-8.5,196.8){\mbox{\tiny $20$}}
\put(41.5,196.8){\mbox{\tiny $21$}}
\put(91.5,196.8){\mbox{\tiny $22$}}
\put(141.5,196.8){\mbox{\tiny $23$}}
\put(191.5,196.8){\mbox{\tiny $24$}}
\put(-8.5,246.8){\mbox{\tiny $25$}}
\put(41.5,246.8){\mbox{\tiny $26$}}
\put(91.5,246.8){\mbox{\tiny $27$}}
\put(-70,-3.2){\mbox{\small start}}
\end{picture}
\hspace{10 mm}
\begin{picture}(200,250)
\multiput(0,0)(50,0){2}{
	\qbezier(0,0)(0,100)(0,250)
	\qbezier(0,0)(25,0)(25,125)
	\qbezier(25,125)(25,250)(50,250)
	\multiput(0,0)(0,50){6}{\ccc}
}
\multiput(150,0)(-50,0){2}{
	\qbezier(0,0)(25,0)(25,100)
	\qbezier(25,100)(25,200)(50,200)
	\qbezier(50,0)(50,100)(50,200)
	\multiput(50,0)(0,50){5}{\ccc}
}
\qbezier(100,0)(100,125)(100,250)
\multiput(100,0)(0,50){6}{\ccc}
\put(-4.5,-3.2){\mbox{\tiny $0$}}
\put(45.5,-3.2){\mbox{\tiny $1$}}
\put(95.5,-3.2){\mbox{\tiny $2$}}
\put(145.5,-3.2){\mbox{\tiny $3$}}
\put(195.5,-3.2){\mbox{\tiny $4$}}
\put(-4.5,46.8){\mbox{\tiny $5$}}
\put(45.5,46.8){\mbox{\tiny $6$}}
\put(95.5,46.8){\mbox{\tiny $7$}}
\put(145.5,46.8){\mbox{\tiny $8$}}
\put(195.5,46.8){\mbox{\tiny $9$}}
\put(-8.5,96.8){\mbox{\tiny $10$}}
\put(41.5,96.8){\mbox{\tiny $11$}}
\put(91.5,96.8){\mbox{\tiny $12$}}
\put(141.5,96.8){\mbox{\tiny $13$}}
\put(191.5,96.8){\mbox{\tiny $14$}}
\put(-8.5,146.8){\mbox{\tiny $15$}}
\put(41.5,146.8){\mbox{\tiny $16$}}
\put(91.5,146.8){\mbox{\tiny $17$}}
\put(141.5,146.8){\mbox{\tiny $18$}}
\put(191.5,146.8){\mbox{\tiny $19$}}
\put(-8.5,196.8){\mbox{\tiny $20$}}
\put(41.5,196.8){\mbox{\tiny $21$}}
\put(91.5,196.8){\mbox{\tiny $22$}}
\put(141.5,196.8){\mbox{\tiny $23$}}
\put(191.5,196.8){\mbox{\tiny $24$}}
\put(-8.5,246.8){\mbox{\tiny $25$}}
\put(41.5,246.8){\mbox{\tiny $26$}}
\put(91.5,246.8){\mbox{\tiny $27$}}
\put(220,-3.2){\mbox{\small start}}
\end{picture}
\end{center}
\caption{Permutations $\phi_1$ (left) and $\phi_2$ (right), with $n=27$, $k=2$ and $b=5$.} 
\label{fig:orderings}
\end{figure}

\begin{lem}
\label{lem:deleting_edges}
Let $n,k,t$ be positive integers, let $X = \{ 0,1, \ldots ,n \}$ and let $\A \subseteq \B(X)$ be a set of $k$ trees. For each $A \in \A$ let $A = A_1 \sqcup A_2 \sqcup A_3 \sqcup \cdots$ be a partition of the vertices of $A$ into at most $t$ connected parts. If for all but $z$ pairs of distinct labels $x,y \in X$ there is at least one $A \in \A$ such that the leaves labelled with $x$ and $y$ are in different parts of $A$, then
	$$ M(\A) \geq n-kt+k-z+1. $$
\end{lem}
\begin{proof}
Let $\{ L_1,L_2, \ldots ,L_m \}$ be an agreement forest of $\A$ and let $F = A/L_1 \sqcup A/L_2 \sqcup \cdots \sqcup A/L_m$ be a forest with leaves labelled by $X$ for some $A \in \A$ (by definition \ref{def:uaf} it doesn't matter which $A \in \A$ we choose). Each tree $A \in \A$ can be divided into the forest $ A = A_1 \sqcup A_2 \sqcup A_3 \sqcup \cdots$ by deleting at most $t-1$ edges. Let these edges be called \emph{hot} edges. There is a total of at most $k(t-1)$ hot edges (at most $t-1$ in each $A \in \A$). An edge in $A$ might not correspond to any edge in $F$, moreover there many be several edges that correspond to the same edge in $F$. In any case, each hot edge corresponds to at most one edge in $F$. Therefore if we delete all the edges in $F$ corresponding to hot edges, then we delete at most $k(t-1)$ edges.\\[2 mm]
The result of the deletion of these edges would be a forest of at least $|X|-z$ components because $z$ is an upper bound on the number of pairs of leaves in the same component. Therefore $F$ must have had $m \geq n+1-z-k(t-1)$ components.
\end{proof}

\begin{lem}
\label{lem:lb_R_uaf}
Let $k \geq 2$, let $n \geq (3k)^k$ and let $X = \{ 0,1, \ldots ,n \}$. If $\T = \{ T_1,T_2, \ldots ,T_k \}$ is a collection of binary trees each with $n+1$ leaves, then there exists $\A = \{ A_1,A_2, \ldots ,A_k \} \subseteq \B(X)$ such that $A_i$ is a labelling of the leaves of $T_i$, and 
	$$ M(\A) \geq n - 3k \sqrt[k]{n} + 1. $$
\end{lem}
\begin{proof}
First set 
	$$ b= \left\lfloor \sqrt[k]{n+1} \right\rfloor \; , \;\; t = \left\lceil \frac{2n}{b^{k-1}-2b^{k-3}} \right\rceil \;\;  \mbox{and} \;\; a = \frac{2n}{t} \leq b^{k-1} - 2b^{k-3}. $$ 
Then divide each $T_i$ into at most $t$ connected parts, $T_i = T_{i1} \sqcup T_{i2} \sqcup T_{i3} \sqcup \cdots$ (using Lemma \ref{lem:divide_tree_equally}) so each part contains less than $a$ leaves. Now it suffices to label the leaves of each $T_i$ such that for any pair of distinct labels $x,y \in X$, there is some $i$ such that $x$ and $y$ are in different parts of $T_i$, and then apply Lemma~\ref{lem:deleting_edges} (with $z=0$). 
\\[3mm]
\begin{figure}

\begin{center}
\setlength{\unitlength}{1pt} 
\begin{picture}(256,125)
	\qbezier(70,125)(70,115)(70,105)
	\qbezier(0,0)(35,52.5)(70,105)
	\qbezier(20,0)(15,7.5)(10,15)
	\qbezier(40,0)(65,37.5)(90,75)
	\qbezier(60,0)(65,7.5)(70,15)
	\qbezier(80,0)(70,15)(60,30)
	\qbezier(120,0)(100,30)(80,60)
	\qbezier(140,0)(105,52.5)(70,105)
	\put(-2,-14){
		\put(0,0){\mbox{$1$}}
		\put(20,0){\mbox{$2$}}
		\put(40,0){\mbox{$3$}}
		\put(60,0){\mbox{$4$}}
		\put(80,0){\mbox{$5$}}
		\put(100,0){\mbox{$6$}}
		\put(120,0){\mbox{$7$}}
		\put(140,0){\mbox{$8$}}
	}
	\put(59,130){\mbox{$0$}}
	\put(70,38){\mbox{$e_1$}}
	\put(89,27){\mbox{$e_2$}}
	\put(0,60){\mbox{$A =$}}
	\put(160,60){\mbox{$A/\{ 4,5,6 \}=$}}	
	\qbezier(240,60)(240,70)(240,80)
	\qbezier(240,60)(232,55)(224,50)
	\put(222,36){\mbox{$5$}}
	\put(254,36){\mbox{$6$}}
	\put(230,85){\mbox{$4$}}
	\color[rgb]{0.0,0.5,0.0} \linethickness{0.5mm} 
	\qbezier(100,0)(105,7.5)(110,15)
	\qbezier(110,15)(95,37.5)(80,60)
	\qbezier(80,60)(70,45)(60,30)
	\qbezier(60,30)(65,22.5)(70,15)
	\qbezier(240,60)(248,55)(256,50)
	\put(0,0){\vertex}
	\put(20,0){\vertex}
	\put(40,0){\vertex}
	\put(60,0){\vertex}
	\put(80,0){\vertex}
	\put(100,0){\vertex}
	\put(120,0){\vertex}
	\put(140,0){\vertex}
	\put(70,125){\vertex}
	\put(240,80){\vertex}
	\put(224,50){\vertex}
	\put(256,50){\vertex}
\end{picture}
\end{center}
\caption{Edges $e_1,e_2$ in $A$ correspond to the same edge in $A/\{ 4,5,6 \}$.}
\label{fig:hotEdges}
\end{figure}
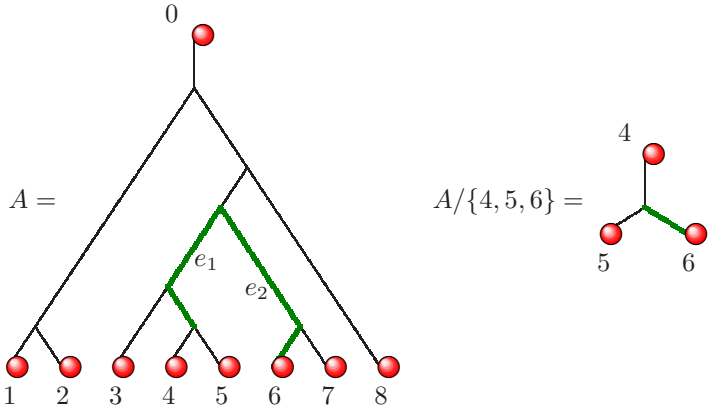
To achieve this: First (by Lemma~\ref{lem:orderings}) find $k$ permutations of $X$ such that for any two distinct labels $x,y \in X$ there is some $1 \leq i \leq k$ such that $|\phi_i(x)-\phi_i(y)| \geq a$. Then assign labels to the leaves of $T_i$ in the order given by $\phi_i$, firstly to all the leaves of $T_{i1}$, then $T_{i2}$, then $T_{i3}$, and so on. For any distinct $x$ and $y$, consider a permutation $\phi_i$ in which $|\phi_i(x)-\phi_i(y)| \geq a$. If some $T_{ij}$ contained both $x$ and $y$, then it would have to also contain the labels 
	$$\{ z : \phi_i(x) > \phi_i(z) > \phi_i(y) \; \mbox{ or } \;  \phi_i(x) < \phi_i(z) < \phi_i(y) \}, $$
but this is impossible since each $T_{ij}$ has size less than $a$. Thus by Lemma~\ref{lem:deleting_edges}, 
	$$ M(\A) \geq n - k(t-1) + 1. $$
For constant $k$, one can easily observe that $t \sim 2\sqrt[k]{n}$ as $n \rightarrow \infty$. For an explicit bound (using $n \geq (3k)^k$), we can show $t \leq 3\sqrt[k]{n} + 1$, in the following manner:
\begin{align*}
b & = \left\lfloor \sqrt[k]{n+1} \right\rfloor > \sqrt[k]{n+1} \left( 1 - \frac{1}{\sqrt[k]{n+1}} \right) \\
& > \sqrt[k]{n} \left( 1 - \frac{1}{3k} \right) \\
\implies b^2-2 & > \sqrt[k]{n^2} \left( \left( 1 - \frac{1}{3k} \right)^2 - \frac{2}{9k^2} \right) \\
& > \sqrt[k]{n^2} \left( 1 - \frac{1}{k} \right) \\
\therefore t & = \left\lceil \frac{2n}{b^{k-3}(b^2-2)} \right\rceil \\
& \leq \left\lceil \frac{2n}{n^{(k-1)/k} \left( 1 - \frac{k}{3k} \right)} \right\rceil \\
& \leq 3n^{1/k}+1.
\end{align*}
\end{proof}

If $\A^\prime = \{ A_1^\prime , \ldots ,A_k^\prime \}$ is obtained from $\A$ by a permutation, $\pi$, of the labels (i.e. The leaf labeled $x$ in $A_i$ is labelled $\pi(x)$ in $A_i^\prime$ for all $x \in X$ and all $i \in [k]$) then $M(\A^\prime) = M(\A)$. Therefore in the above Lemma, we could have constructed $\A$ so that one of the trees has a particular labelling, say $A_1$. Now if we set $k=2$ and apply the unrooted agreement forest lemma (\ref{lem:uafl}) then 
	$$ R_{TBR}(n) \geq n - 6 \sqrt{n}, $$
for all integers $n \geq 36$.\footnote{For $n<36$ this bound is trivial since $R_{TBR}(n) \geq 0 > n - 6\sqrt{n}$.} This is a suitable lower bound for Theorem~\ref{thm:extreme}.

\section{Lower Bound for the Expectation}

\begin{lem}
\label{lem:lb_expect_uaf}
Let $k \geq 2$, let $n$ be a positive integer, let $X = \{ 0,1, \ldots ,n \}$, and let $T_1,T_2, \ldots ,T_k$ be binary trees each with $n+1$ leaves. If $\A = \{ A_1,A_2, \ldots ,A_k \} \subseteq \B(X)$ is chosen uniformly at random such that each $A_i$ is a labelling of the leaves of $T_i$, then 
	$$ \E[M(\A)] > n - (2k+1)n^{2/(k+1)} + 1. $$
\end{lem}
\begin{proof}
First set 
	$$ a = (n+1)^{\frac{k-1}{k+1}} \;\; \mbox{and} \;\; t = \left\lceil 2(n+1)^{2/(k+1)} \right\rceil. $$ 
Then, using Lemma~\ref{lem:divide_tree_equally}, divide each $A_i$ into $t$ (possibly empty) connected parts, so that each part has less than $a$ leaves: $A_i = A_{i1} \sqcup \cdots \sqcup A_{it}$. For distinct $x,y \in X$, notice that
	$$ \P( y \in L(A_{ij}) | x \in L(A_{ij})) = \frac{|L(A_{ij})|-1}{n} < \frac{a-1}{n}. $$
Now let $Z$ be the set of pairs of distinct labels $\{ x,y \}$ such that for each $i$, $x$ and $y$ are in the same part $A_{ij}$. Explicitly 
	$$ Z := \big\{ \{ x,y \} \; : \; x \not= y \mbox{ and } \forall i \; \exists j \; \mbox{ s.t. } \; x,y \in L(A_{ij}) \big\}.$$ 
Since the labellings are chosen independently, $\P( \{ x,y \} \in Z ) < \left( \frac{a-1}{n} \right)^k$ for all $x,y \in X$. Now let $z = |Z|$. By the linearity of expectation 
	$$ \E[z] = \binom{n+1}{2} \P( \{ x,y \} \in Z) < n^2 \left( \frac{a-1}{n} \right)^k < n^{2/(k+1)}. $$
Since $t-1 \leq 2n^{2/(k+1)}$ and $\E[z] \leq n^{2/(k+1)}$, we can  applying Lemma~\ref{lem:deleting_edges} to get 
	\begin{align*} \E[M(\A)] 
	& \geq n - k(t-1) - \E[z] + 1\\ 
	& > n - (2k+1)n^{2/(k+1)} + 1.\tag*{\qedhere}
	\end{align*}
\end{proof}

Setting $k=2$ in Lemma \ref{lem:lb_expect_uaf} and choosing $T_1$ and $T_2$ uniformly and independantly, (so that $\A = \{ A_1,A_2 \}$ is distributed uniformly) gives $\E[M(A_1,A_2)] \geq n-5n^{2/3}+1$. Applying the unrooted agreement forest lemma (\ref{lem:uafl}) gives
	$$ \E[d_{TBR}(A,B)] \geq n - 5n^{2/3}. $$
which is a suitable lower bound for Theorem~\ref{thm:expectation}.

\vspace{3 mm}

\noindent
\textbf{Acknowledgement:} We would like to thank Charles Semple for helpful comments.


\begin{thebibliography}{99} 

\bibitem{allen1998}
B. Allen: 
Subtree transfer operations and their induced metrics on evolutionary trees, 
MSc Thesis, University of Canterbury, Christchurch, New Zealand, 1998.

\bibitem{allen2001}
B. Allen, M. Steel: 
Subtree Transfer Operations and Their Induced Metrics on Evolutionary Trees, 
Ann. Comb. \mbox{\bf{5}} (2001) 1--15. 

\bibitem{bordewich}
M. Bordewich, C. Semple: 
On the Computational Complexity of the Rooted Subtree Prune and Regraft Distance, 
Ann. Comb. \mbox{\bf{8}} (2004) 409--423.

\bibitem{ding}
Y. Ding, S. Grunewald, P.J. Humphries:
On agreement forests, 
J. Comb. Theory, Ser. A 118, \mbox{\bf 7} (2011) 2059--2065.

\bibitem{felsenstein}
J. Felsenstein: 
Inferring Phylogenies,  
Sinauer Associates, Sunderland MA, (2003).

\bibitem{golobof}
P. Golobof: 
Calculating SPR distances between trees, 
Cladistics \mbox{\bf{23}} (2007) 1--7.

\bibitem{hein1990}
J. Hein: 
Reconstructing evolution of sequences subject to recombination using parsimony,
Math. Biosci. \mbox{\bf{98}} (1990) 185--200.

\bibitem{hein1996}
J. Hein, T. Jiang, L. Wang, K. Zhang:
On the complexity of comparing evolutionary trees, 
Disc. Appl. Math. \mbox{\bf{71}} (1996) 153--169.

\bibitem{hillis}
D. Hillis, G. Olsen, D. Swofford, P. Waddell:
Molecular systematics, 
Sinauer Associates, Sunderland MA, (1996) 407--514.

\bibitem{martin}
D. Martin, B. Thatte:
The maximum agreement subtree problem,
Disc. Appl. Math. \mbox{\bf{161}} (2013) 1805--1817.

\bibitem{song}
Y. Song: 
On the Combinatorics of Rooted Binary Phylogenetic Trees, 
Ann. Comb. \mbox{\bf{7}} (2003) 365--379.

\bibitem{swofford1998}
D. Swofford:
PAUP*, Phylogenetic Analysis using Parsimony (*and other methods),
Sinauer Associates, Sunderland MA, (1998).

\bibitem{zwickl}
D. Zwickl:
Genetic algorithm approaches for the phylogenetic analysis of large biological sequence datasets under the maximum likelihood criterion,
Ph.D. Thesis, University of Texas, USA, (2006).

\end{thebibliography}
\end{document}